\documentclass[11pt]{article}

\usepackage{amsmath,amssymb,amsthm,mathtools}
\usepackage[a4paper,margin=1in]{geometry}

\newtheorem{theorem}{Theorem}
\newtheorem{lemma}[theorem]{Lemma}
\newtheorem{conjecture}[theorem]{Conjecture}

\numberwithin{equation}{section}

\newcommand{\OSOME}{\operatorname{OSOME}}
\newcommand{\OSOMEo}{\operatorname{OSOME}_{o}}
\newcommand{\OSOMEe}{\operatorname{OSOME}_{e}}
\newcommand{\SOME}{\operatorname{SOME}}
\newcommand{\DSOME}{\operatorname{DSOME}}
\newcommand{\SONO}{\operatorname{SONO}}
\newcommand{\SENO}{\operatorname{SENO}}
\newcommand{\SNO}{\operatorname{SNO}}

\title{A proof of the Baruah--Gogoi conjecture on sums of
odd and even overlined parts}
	 
{\small 
	\author{ \normalsize  Eric H. Liu$^1$, X. L. Liu$^2$  and   Olivia X. M. Yao$^{3}$ \\[6pt]
		\small	$^{1}$School of Statistics and
		Information,\\
		\small	Shanghai University of International
		\small	Business and Economics,\\
		\small	Shanghai, 201620, P. R. China\\	
		\small	$^{2,3}$School of Mathematical Sciences\\
		\small	Suzhou University of Science and Technology\\
		\small	Suzhou 215009, Jiangsu, P. R. China\\
		\small	Emails: \texttt{  
			$^{1}$liuhai@suibe.edu.cn, $^{2}$luluux@163.com, 
			$^{3}$yaoxiangmei@163.com }
	}
}
\date{}

\begin{document}

\maketitle

\begin{abstract}
	Andrews and   Dastidar introduced the statistic $\SOME(n)$,
	defined as the difference between the sum of all odd parts and the sum
	of all even parts occurring in the partitions of $n$, and investigated
	its arithmetic properties. Motivated by their work, Baruah and Gogoi
	defined two analogous statistics for overpartitions, namely
	$\OSOMEo(n)$ and $\OSOMEe(n)$, which record the sums of all odd and all
	even overlined parts, respectively. They established several
	congruences for these statistics and their difference, and proposed
	 a conjecture on  congruences modulo $7$. In this paper, we prove their conjecture
	by elementary $q$-series methods. Our proof combines  
	two specializations of Watson's quintuple product identity, a
	$(p,k)$-parameterization of divisor-sum series, and termwise
	logarithmic differentiation.  
\end{abstract}

\smallskip
\noindent\textbf{Keywords.}  congruences; partitions; overpartitions; 
sums of overlined parts; quintuple product identity.
 
\par\smallskip
\noindent\textbf{2020 Mathematics Subject Classification.}
Primary 11P83; Secondary 05A17.

\section{Introduction}

An integer partition of a positive integer \(n\) is a nonincreasing
finite sequence of positive integers whose sum is \(n\); see
\cite{Andrews1976}.  For example, the seven partitions of \(5\) are
\[
\begin{gathered}
(5),\ (4,1),\ (3,2),\ (3,1,1),\ 
(2,2,1),\ (2,1,1,1),\ (1,1,1,1,1).
\end{gathered}
\]
Let \(p(n)\) denote the number of partitions of \(n\), with
\(p(0):=1\).  Throughout the paper, \(|q|<1\), and we write
\[
(a;q)_\infty:=\prod_{j\geq0}(1-aq^j),
\qquad
J_m:=(q^m;q^m)_\infty.
\]
Here $m$ is a positive integer. 
Euler's generating function for $p(n)$ is
\[
\sum_{n\geq0}p(n)q^n=\frac1{J_1}.
\]

Among the best-known arithmetic properties of \(p(n)\) are
Ramanujan's congruences \cite{Ramanujan}.  The congruence relevant
here is
\begin{align}\label{1-1}
 p(5n+4)\equiv0\pmod5.
\end{align}
Recently, Andrews and Dastidar \cite{AndrewsGhoshDastidar}
obtained refinements of Ramanujan's congruence modulo \(5\).
They defined the statistic \(\SOME(n)\) as the sum of all odd
parts occurring in the partitions of \(n\) minus the corresponding
sum of all even parts. 
   For instance,
the seven partitions displayed above contain odd parts with total
weight \(23\) and even parts with total weight \(12\); hence
\(\SOME(5)=11\).  They proved
\[
\sum_{n\geq0}\SOME(n)q^n
=\frac1{J_1}\sum_{n\geq1}\frac{q^n}{(1+q^n)^2}.
\]
They
established, among other results,
\begin{align*}
 \SOME(4n)&\equiv0\pmod4,\\
 \SOME(5n+2)&\equiv0\pmod5,
\end{align*}
and
\begin{align}\label{1-2}
 \SOME(5n+4)\equiv0\pmod5.
\end{align}
Combining \eqref{1-1} and \eqref{1-2}, one sees that both the odd-part
total and the even-part total over all partitions of \(5n+4\) are
divisible by \(5\).  
    Andrews and Dastidar \cite{AndrewsGhoshDastidar} conjectured  that 
\[
\operatorname{SOME}(\lambda)\equiv 0 \pmod{5^\alpha},
\]
where $\alpha\geq 1$ and $\lambda\geq 0$ are integers satisfying
$
24\lambda\equiv 1\pmod{5^\alpha}.
$
 This conjecture was proved by  Bardhan and  Saikia \cite{BardhanSaikia}.
  Yao and Zhou \cite{Yao} proved that, for every \(n\geq0\), 
\[
\SOME(25n+2)\equiv0\pmod {25}.
\]
Andrews and   Dastidar \cite{AndrewsGhoshDastidar} also introduced the
corresponding statistic \(\DSOME(n)\) for partitions into distinct
parts, whose arithmetic properties were further investigated by
Baruah and Gogoi \cite{BaruahGogoiDSOME}.

An \emph{overpartition} is a partition in which the first occurrence
of each distinct part may be overlined.  Corteel and Lovejoy
\cite{CorteelLovejoy} developed the systematic theory of
overpartitions.  If \(\overline p(n)\) is the number of
overpartitions of \(n\), then
\begin{align}
\sum_{n\geq0}\overline p(n)q^n
&:=\prod_{m\geq1}\frac{1+q^m}{1-q^m}
=\frac{J_2}{J_1^2}. \label{overpartition}
\end{align}

Gireesh and Hemanthkumar \cite{GireeshHemanthkumar} introduced an
overpartition analogue of \(\SOME(n)\) and proved congruences modulo
\(3\), \(5\), and certain powers of \(2\).  Garvan and Sarma
\cite{GarvanSarma} considered the total odd, even, and unrestricted
non-overlined parts in all overpartitions of \(n\), denoted by
\(\SONO(n)\), \(\SENO(n)\), and \(\SNO(n)\), respectively.  Their
theta function and dissection methods yield, among other congruences,
\begin{align*}
\SONO(5n+3)&\equiv0\pmod5,\\
\SENO(7n+5)&\equiv0\pmod7,\\
\SNO(5n+2)&\equiv\ \SNO(5n+4)\equiv0\pmod5,\\
\SNO(7n+3)&\equiv0\pmod7.
\end{align*}

Motivated by these developments, Baruah and Gogoi
\cite{BaruahGogoi} defined two statistics. The first,
\(\OSOMEo(n)\), is the sum of all odd overlined parts in all
overpartitions of \(n\); the second, \(\OSOMEe(n)\), is the
corresponding sum of all even overlined parts.
 They then put
\[
 \OSOME(n):=\OSOMEo(n)-\OSOMEe(n).
\]
For example, the eight overpartitions of $3$ are
\[
\begin{gathered}
	(3),\quad (\overline{3}),\quad
	(2,1),\quad (\overline{2},1),\quad 
	(2,\overline{1}),\quad
	(\overline{2},\overline{1}),\quad
	(1,1,1),\quad (\overline{1},1,1).
\end{gathered}
\]
The odd overlined parts have total weight
$
3+1+1+1=6,
$
whereas the even overlined parts have total weight
$
2+2=4.
$
Consequently,
\[
\OSOMEo(3)=6,\quad
\OSOMEe(3)=4,\quad
\OSOME(3)=\OSOMEo(3)-\OSOMEe(3)=2.
\] 
They obtained closed generating functions and congruences modulo
\(3\), \(5\), \(8\), and \(11\).  Their computations led to the
following five congruences modulo \(7\).

\begin{conjecture}  [{Baruah and Gogoi \cite[Conjecture~7.3]{BaruahGogoi}}]
	\label{conj-1}
For every $n\geq0$,
\begin{align}
 \OSOME(56n+43)&\equiv0\pmod7, \label{c-1}\\
 \OSOMEo(56n+11)&\equiv0\pmod7,\label{c-2}\\
 \OSOMEo(56n+23)&\equiv0\pmod7,\label{c-3}\\
 \OSOMEo(56n+47)&\equiv0\pmod7,\label{c-4}\\
 \OSOMEe(56n+51)&\equiv0\pmod7.\label{c-5}
\end{align}
\end{conjecture}

The aim of this paper is to prove Conjecture~\ref{conj-1} by
 combining two specializations of the quintuple product
identity, a \((p,k)\)-parameterization of divisor-sum series, and
formal differentiation.

\section{Preliminaries}

Let
\[
 \phi(q):=\sum_{r\in\mathbb Z}q^{r^2},
 \qquad
 \psi(q):=\sum_{r\geq0}q^{r(r+1)/2}.
\]
The Jacobi triple product gives the standard product evaluations
\cite[Chapter~1]{Berndt-1}
\begin{equation}\label{eq:theta-products}
 \phi(-q)=\frac{J_1^2}{J_2},
 \qquad
 \phi(q)=\frac{J_2^5}{J_1^2J_4^2},
 \qquad
 \psi(q)=\frac{J_2^2}{J_1}.
\end{equation}

 Baruah and Gogoi \cite{BaruahGogoi} established the generating functions for 
  $\OSOME(n)$, $\OSOMEo(n)$ and  $\OSOMEe(n)$: 
 \begin{align}
 \sum_{n\geq0}\OSOME(n)q^n
 &=\frac18\left(\frac1{\phi(-q)}-\phi(-q)^3\right),
 \label{eq:OSOME-gf}\\
 \sum_{n\geq0}\OSOMEo(n)q^n
 &=\frac{\phi(q)^4-2\phi(-q)^4+1}{24\phi(-q)},
 \label{eq:OSOMEo-gf}\\
 \sum_{n\geq0}\OSOMEe(n)q^n
 &=\frac{\phi(q)^4+\phi(-q)^4-2}{24\phi(-q)}.
 \label{eq:OSOMEe-gf}
\end{align}

Comparing coefficients in \eqref{eq:OSOME-gf},
\eqref{eq:OSOMEo-gf}, and \eqref{eq:OSOMEe-gf} gives
\begin{align}
 8\OSOME(n)&=\overline p(n)-c(n),\label{eq:OSOME}\\
 24\OSOMEo(n)&=x(n)-2c(n)+\overline p(n),\label{eq:OSOMEo}\\
 24\OSOMEe(n)&=x(n)+c(n)-2\overline p(n),\label{eq:OSOMEe}
\end{align}
where $\overline{p}(n)$ is defined by \eqref{overpartition} and 
 $c(n)$ and $x(n)$  are defined by
\begin{align}
\sum_{n\geq0}c(n)q^n&:=\phi(-q)^3, \label{c}
\\
\sum_{n\geq0}x(n)q^n&:=\frac{\phi(q)^4}{\phi(-q)}. \label{x}
\end{align}

\section{Some lemmas}

 Unless otherwise stated, all congruences    in this paper are taken modulo $7$.
 
Put
\begin{align*}
 P_3(q)&:=\sum_{n\geq0}\overline p(8n+3)q^n,&
 C_3(q)&:=\sum_{n\geq0}c(8n+3)q^n,\\
 X_3(q)&:=\sum_{n\geq0}x(8n+3)q^n,&
 P_7(q)&:=\sum_{n\geq0}\overline p(8n+7)q^n,\\
 X_7(q)&:=\sum_{n\geq0}x(8n+7)q^n.
\end{align*}

\begin{lemma}\label{lem:dissections}
One has
\begin{align}
 P_3(q)&\equiv
 \frac{J_2^{34}}{J_1^{27}J_4^8}
 +2q\frac{J_2^{10}J_4^8}{J_1^{19}},\label{eq:P3}\\
 C_3(q)&\equiv-\frac{J_2^6}{J_1^3},\label{eq:C3}\\
 X_3(q)&\equiv
 \frac{J_2^{54}}{J_1^{35}J_4^{16}}
 +4q^2\frac{J_2^6J_4^{16}}{J_1^{19}},\label{eq:X3}\\
 P_7(q)&\equiv\frac{J_2^{22}}{J_1^{23}},\label{eq:P7}\\
 X_7(q)&\equiv
 4\frac{J_2^{42}}{J_1^{31}J_4^8}
 +q\frac{J_2^{18}J_4^8}{J_1^{23}}.\label{eq:X7}
\end{align}
\end{lemma}

\begin{proof} For a formal power series $F(q)$, if 
  $F(q)=F_0(q^2)+qF_1(q^2)$, then its even and odd coefficient
generating functions are $F_0(q)$ and $F_1(q)$, respectively.  We use
this elementary extraction repeatedly.

 Hirschhorn and Sellers \cite{Hirschhorn-1} proved that 
\begin{equation}\label{eq:barp-four-three}
 \sum_{n\geq0}\overline p(4n+3)q^n
 =8\frac{J_2J_4^6}{J_1^8}.
\end{equation}
It follows from \cite[(2.11)]{Xia-Yao-2013} that 
\begin{equation}\label{eq:two-dissection}
	\frac1{J_1^4}
	=\frac{J_4^{14}}{J_2^{14}J_8^4}
	+4q\frac{J_4^2J_8^4}{J_2^{10}}.
\end{equation}
Squaring \eqref{eq:two-dissection} gives
\begin{equation}\label{eq:J1-eight-dissection}
 \frac1{J_1^8}
 =\frac{J_4^{28}}{J_2^{28}J_8^8}
 +8q\frac{J_4^{16}}{J_2^{24}}
 +16q^2\frac{J_4^4J_8^8}{J_2^{20}}.
\end{equation}
Substituting \eqref{eq:J1-eight-dissection}
  into \eqref{eq:barp-four-three} gives
\[
 \sum_{n\geq0}\overline p(4n+3)q^n
 =8\frac{J_4^{34}}{J_2^{27}J_8^8}
 +64q\frac{J_4^{22}}{J_2^{23}}
 +128q^2\frac{J_4^{10}J_8^8}{J_2^{19}}.
\]
The first and third terms contain only even powers of $q$, while the
middle term contains only odd powers.  Thus
\begin{align*}
 P_3(q)&=8\frac{J_2^{34}}{J_1^{27}J_4^8}
       +128q\frac{J_2^{10}J_4^8}{J_1^{19}},\qquad 
 P_7(q) =64\frac{J_2^{22}}{J_1^{23}}.
\end{align*}
Reducing the numerical coefficients modulo $7$ proves
\eqref{eq:P3} and \eqref{eq:P7}.

It follows from \cite[Entry 25 (i) and (ii), p. 40]{Berndt1991} that 
\[
 \phi(-q)=\phi(q^4)-2q\psi(q^8). 
\]
Therefore, 
\begin{align}\label{3-12}
 \phi(-q)^3
 ={}&\phi(q^4)^3-6q\phi(q^4)^2\psi(q^8)
 +12q^2\phi(q^4)\psi(q^8)^2-8q^3\psi(q^8)^3.
\end{align}
The four terms are supported on exponents congruent to $0$, $1$, $2$,
and $3$ modulo $4$, respectively.  In light of \eqref{c} and \eqref{3-12},
\begin{equation}\label{eq:c-four-three}
 \sum_{n\geq0}c(4n+3)q^n
 =-8\psi(q^2)^3=-8\frac{J_4^6}{J_2^3}.
\end{equation}
Here the last equality follows from the product formula for
\(\psi(q)\) in \eqref{eq:theta-products}.
The right side of \eqref{eq:c-four-three} contains only even powers.
Consequently,
\begin{equation}\label{eq:c-eight-seven}
	c(8n+7)=0 
\end{equation} 
and 
\[
 C_3(q)=-8\frac{J_2^6}{J_1^3},
\]
which proves \eqref{eq:C3} modulo $7$.

It remains to treat $x(n)$.

It follows from \cite[Entry 25 (i), (ii) and (iii), p. 40]{Berndt1991} that
\begin{align}
\phi(q)\phi(-q)&=\phi(-q^2)^2,
 \label{3-13}\\ 
\phi(q)&=\phi(q^4)+2q\psi(q^8) \label{3-14}. 
\end{align}
By \eqref{3-13} and \eqref{3-14}, 
\begin{align}\label{m-1}
\frac1{\phi(-q)}
=\frac{\phi(q^4)+2q\psi(q^8)}{\phi(-q^2)^2}.
\end{align}
It follows from \cite[Entry 25   (vii), p. 40]{Berndt1991} that
\begin{align}\label{t-1}
 \phi(q)^4=\phi(-q)^4+16q\psi(q^2)^4.
\end{align}
Combining \eqref{x} and \eqref{t-1}  yields 
\begin{align}\label{m-2}
 \sum_{n\geq0}x(n)q^n
 &=\phi(-q)^3+16q\frac{\psi(q^2)^4}{\phi(-q)} . 
\end{align}
Substituting  \eqref{3-12} and \eqref{m-1} into \eqref{m-2},
   extracting the odd powers, dividing by \(q\), and replacing \(q^2\) by \(q\), we obtain 
\begin{align}
	\sum_{n\geq0}x(2n+1)q^n
	&=-6\phi(q^2)^2\psi(q^4)-8q\psi(q^4)^3 +16\frac{\psi(q)^4 \phi(q^2)}{\phi(-q)^2} \nonumber\\
	&\equiv \frac{2}{J_1^8}\cdot \frac{J_2^8J_4^5}{J_8^2}
	+\frac{J_4^9}{J_2^4J_8^2}
	-q\frac{J_8^6}{J_4^3} . \label{m-3}
\end{align}
Substituting \eqref{eq:J1-eight-dissection} into \eqref{m-3}, extracting the odd powers, dividing by \(q\), and replacing \(q^2\) by \(q\), we get
\begin{equation}\label{eq:x-four-three-exact}
 \sum_{n\geq0}x(4n+3)q^n
   \equiv \frac{J_2^{21}}{J_4^2}
   \left(\frac{2}{J_1^{16}}-\frac{J_4^8}{J_2^{24}}\right)
  .
\end{equation}
Replacing $q$ by $q^2$ in \eqref{t-1}   gives 
\[
\frac{J_4^{20}}{J_2^8J_8^8}
=
\frac{J_2^8}{J_4^4}
+16q^2\frac{J_8^8}{J_4^4}.
\]
Multiplying both sides of the above identity by
\(J_4^8/J_2^{20}\) yields
\begin{align}
 \left(\frac{J_4^{14}}{J_2^{14}J_8^4}\right)^2
 -16q^2\left(\frac{J_4^2J_8^4}{J_2^{10}}\right)^2
 =\frac{J_4^4}{J_2^{12}}.\label{r-2}
\end{align}
Squaring both sides of \eqref{r-2} and reducing the numerical coefficients modulo
$7$, we obtain
\begin{align}
	\frac{J_4^8}{J_2^{24}}
	\equiv{}&
	\left(\frac{J_4^{14}}{J_2^{14}J_8^4}\right)^4+3q^2
	\left(\frac{J_4^{14}}{J_2^{14}J_8^4}\right)^2
	\left(\frac{J_4^2J_8^4}{J_2^{10}}\right)^2+4q^4
	\left(\frac{J_4^2J_8^4}{J_2^{10}}\right)^4
	 .\label{r-3}
\end{align} 
In view of \eqref{eq:two-dissection} and \eqref{r-3}, 
\begin{align}\label{m-4}
\frac{2}{J_1^{16}}-\frac{J_4^8}{J_2^{24}}&\equiv 2\left(
 \frac{J_4^{14}}{J_2^{14}J_8^4}
 +4q\frac{J_4^2J_8^4}{J_2^{10}}
 \right)^4 
 \nonumber\\
 &\qquad - \left(	\left(\frac{J_4^{14}}{J_2^{14}J_8^4}\right)^4+3q^2
 \left(\frac{J_4^{14}}{J_2^{14}J_8^4}\right)^2
 \left(\frac{J_4^2J_8^4}{J_2^{10}}\right)^2+4q^4
 \left(\frac{J_4^2J_8^4}{J_2^{10}}\right)^4 \right)  \nonumber\\
&\equiv
 \left(\frac{J_4^{14}}{J_2^{14}J_8^4}\right)^4
 +4q\left(\frac{J_4^{14}}{J_2^{14}J_8^4}\right)^3
       \left(\frac{J_4^2J_8^4}{J_2^{10}}\right) \nonumber \\
&\qquad
 +q^3\left(\frac{J_4^{14}}{J_2^{14}J_8^4}\right)
       \left(\frac{J_4^2J_8^4}{J_2^{10}}\right)^3
 +4q^4\left(\frac{J_4^2J_8^4}{J_2^{10}}\right)^4
 .
\end{align}
Substituting \eqref{m-4} into \eqref{eq:x-four-three-exact} gives 
\begin{align}\label{eq:x-four-dissection}
 \sum_{n\geq0}x(4n+3)q^n
 \equiv{}&\frac{J_4^{54}}{J_2^{35}J_8^{16}}
 +4q\frac{J_4^{42}}{J_2^{31}J_8^8}  +q^3\frac{J_4^{18}J_8^8}{J_2^{23}}
 +4q^4\frac{J_4^6J_8^{16}}{J_2^{19}}
  .
\end{align}
Extracting the even powers from \eqref{eq:x-four-dissection} proves
\eqref{eq:X3}; extracting the odd powers proves \eqref{eq:X7}.
\end{proof}

\begin{lemma}\label{lem:sparse}
For $n\geq 0$, 
\begin{equation}\label{eq:sparse-conclusion}
 x(8(7n+r)+3)\equiv0\pmod7
 \qquad(r=1,4,5,6).
\end{equation}
\end{lemma}

\begin{proof}
We use the dissection formula \eqref{eq:X3} from
Lemma~\ref{lem:dissections}.
The binomial theorem gives, for every positive integer $m$,
\begin{equation}\label{eq:freshman}
 J_m^7\equiv J_{7m} .
\end{equation}
Applying \eqref{eq:freshman} to \eqref{eq:X3}, we find
\begin{equation}\label{eq:X3-sparse}
 X_3(q)\equiv
 \frac{J_{98}}{J_7^5J_{28}^2} \cdot \frac{J_2^5}{J_4^2}
 +4q^2\frac{J_{14}J_{28}^2}{J_7^3} \cdot 
       \frac{J_1^2J_4^2}{J_2}.
\end{equation}
The following two specializations of Watson's quintuple product
identity \cite{Watson} are recorded in
\cite[(1.3.60) and (1.3.61)]{Berndt-1}:
\begin{align}
 \frac{J_1^5}{J_2^2}
 &=\sum_{m\in\mathbb Z}(6m+1)q^{m(3m+1)/2},
 \label{eq:quintuple-one}\\
 \frac{J_1^2J_4^2}{J_2}
 &=\sum_{m\in\mathbb Z}(3m+1)q^{m(3m+2)}.
 \label{eq:quintuple-two}
\end{align}
Use \eqref{eq:quintuple-one} with $q$ replaced by $q^2$ in the first
term of \eqref{eq:X3-sparse}, and use \eqref{eq:quintuple-two} in the
second term.  This gives
\[
\begin{split}
 X_3(q)\equiv{}&
 \frac{J_{98}}{J_7^5J_{28}^2}
 \sum_{m\in\mathbb Z}(6m+1)q^{m(3m+1)} 
  +4q^2\frac{J_{14}J_{28}^2}{J_7^3}
 \sum_{m\in\mathbb Z}(3m+1)q^{m(3m+2)}.
\end{split}
\]
The product factors outside the sums contain only powers of $q^7$.
Checking $m=7k+r$ with $k\in \mathbb{Z}$
 and $r=0,1,\ldots,6$  shows that the first sum has nonzero
coefficients modulo $7$ only in exponent classes $0$, $2$, and $3$.
After the factor $q^2$ is included, the same is true of the second
sum.  This proves  
\eqref{eq:sparse-conclusion}.
\end{proof}

  For the remainder of the paper, put
\begin{equation}\label{eq:bt}
 b:=\frac{J_1^8}{J_2^4}=\phi(-q)^4,
 \qquad
 t:=q\frac{J_4^8}{J_2^4}=q\psi(q^2)^4.
\end{equation}
Define
\begin{align*}
 L(q)&:=1-24\sum_{n\geq1}\sigma_1(n)q^n,\\
 M(q)&:=1+240\sum_{n\geq1}\sigma_3(n)q^n,\\
 N(q)&:=1-504\sum_{n\geq1}\sigma_5(n)q^n,
\end{align*}
where
\[
\sigma_r(n):=\sum_{d|n,\ d>0} d^r.
\]

\begin{lemma}\label{lem:pk}
One has
\begin{align}
 M(q^2)&=b^2+16bt+256t^2,\label{eq:M-bt}\\
 N(q^2)&=(b+32t)(b+8t)(b-16t),\label{eq:N-bt}\\
 2L(q^2)-L(q)&=b+32t.\label{eq:L-bt}
\end{align}
\end{lemma}

\begin{proof}
Let
\begin{equation*} 
	p=p(q):=
	\frac{\phi(q)^2-\phi(q^3)^2}
	{2\phi(q^3)^2},
	\qquad
	k=k(q):=
	\frac{\phi(q^3)^3}{\phi(q)}.
\end{equation*}
These are the \((p,k)\)-parameters introduced by Alaca and
Williams \cite[Section~3]{AlacaWilliams}.  
 In their notation, equations
(3.66)--(3.68) give
\begin{equation}\label{eq:pk-bt}
 b=(1+p)(1-p)^3k^2,
 \qquad
 16t=p(2+p)^3k^2.
\end{equation}
Their equations (3.70), (3.74), and (3.84), with the sign in (3.84)
reversed to match the present order, give
\begin{align}
 M(q^2)={}&(1+4p+64p^2+178p^3+235p^4+178p^5
 +64p^6+4p^7+p^8)k^4, \label{m-5}\\
 N(q^2)={}&\bigl(1+6p-114p^2-625p^3-\tfrac{4059}{2}p^4
 -4302p^5-5556p^6 \nonumber \\
 &\qquad{}-4302p^7-\tfrac{4059}{2}p^8-625p^9
 -114p^{10}+6p^{11}+p^{12}\bigr)k^6, \label{m-6}\\
 2L(q^2)-L(q)={}&(1+14p+24p^2+14p^3+p^4)k^2.  \label{m-7}
\end{align}
Substitution of \eqref{eq:pk-bt} gives, by direct multiplication,
\begin{align*}
b^2+16bt+256t^2
={}&(1+4p+64p^2+178p^3+235p^4+178p^5
+64p^6+4p^7+p^8)k^4,\\
(b+32t)(b+8t)(b-16t)
={}&\bigl(1+6p-114p^2-625p^3-\tfrac{4059}{2}p^4 
-4302p^5-5556p^6 \\
&\qquad{}-4302p^7
-\tfrac{4059}{2}p^8-625p^9-114p^{10}+6p^{11}+p^{12}\bigr)k^6,\\
b+32t
={}&(1+14p+24p^2+14p^3+p^4)k^2.
\end{align*}
Comparison with \eqref{m-5}, \eqref{m-6}, and \eqref{m-7}
proves \eqref{eq:M-bt}--\eqref{eq:L-bt}.
\end{proof}

We shall use the differential operator  
\[
\Theta:=q\frac{d}{dq}.
\]

\begin{lemma}\label{lem:logarithmic}
One has
\begin{align}
 \Theta\log(t/b)=\phi(q)^4=b+16t.\label{vv-0}
\end{align}
\end{lemma}

\begin{proof}
	Termwise logarithmic differentiation gives
	\begin{equation}\label{eq:log-J}
		\Theta\log J_m=\frac{m}{24}\bigl(L(q^m)-1\bigr).
	\end{equation}
Since $t/b=qJ_4^8/J_1^8$, equation \eqref{eq:log-J} gives 
\begin{align*}
 \Theta\log(t/b)
 &=1+8\Theta\log J_4-8\Theta\log J_1\\
 &=1+8\sum_{n\geq1}\frac{nq^n}{1-q^n}
   -32\sum_{n\geq1}\frac{nq^{4n}}{1-q^{4n}}\\
 & 
 =\phi(q)^4.
\end{align*}
The last equality  follows from \cite[(3.3.7)]{Berndt-1}.  Finally,  \eqref{t-1} and \eqref{eq:bt} give
$\phi(q)^4=b+16t$.
\end{proof}

\begin{lemma}\label{lem:differential-N}
One has
\begin{equation}\label{eq:differential-N}
 \Theta N(q^2)=L(q^2)N(q^2)-M(q^2)^2.
\end{equation}
\end{lemma}

\begin{proof} 
We use the parameter identities of Lemma~\ref{lem:pk} and the
logarithmic differentiation formula established in
Lemma~\ref{lem:logarithmic}.
 Since
$
b=\frac{J_1^8}{J_2^4} 
$,    equation \eqref{eq:log-J} gives
\begin{align}
	\Theta\log b
	&=
	8\Theta\log J_1-4\Theta\log J_2 \nonumber \\
	&=
	\frac{8}{24}\bigl(L(q)-1\bigr)
	-\frac{8}{24}\bigl(L(q^2)-1\bigr) \nonumber \\
	&=
	\frac13\bigl(L(q)-L(q^2)\bigr). \label{m-8}
\end{align}
In view of \eqref{eq:L-bt}, 
 \[
L(q)-L(q^2)=L(q^2)-b-32t.
\]
Combining this identity with \eqref{m-8}, we obtain 
\[
\Theta\log b
=
\frac13\bigl(L(q^2)-b-32t\bigr).
\]
Since
$\Theta b=b\,\Theta\log b,
$
it follows that
\begin{align}\label{m-9}
	\Theta b
	=
	\frac{b}{3}\bigl(L(q^2)-b-32t\bigr).
\end{align}

Since
\[
\log t=\log b+\log\frac{t}{b},
\]
we obtain
\begin{align*}
	\Theta\log t
	&=
	\Theta\log b+\Theta\log\frac{t}{b}\\
	&=
	\frac13\bigl(L(q^2)-b-32t\bigr)+b+16t\qquad ({\rm by}\ \eqref{vv-0})\\
	&=
	\frac13\bigl(L(q^2)+2b+16t\bigr).
\end{align*}
Finally,
$
\Theta t=t\,\Theta\log t,
$
and hence
\begin{align}
	\Theta t
	=
	\frac{t}{3}\bigl(L(q^2)+2b+16t\bigr).\label{m-10}
\end{align}

  On expanding
\eqref{eq:N-bt} as a polynomial in $b$ and $t$ and differentiating,
we find
\begin{align*}
 \Theta N(q^2)
 ={}&(3b^2+48bt-384t^2)\Theta b+(24b^2-768bt-12288t^2)\Theta t.
\end{align*}
Substitution of \eqref{m-9} and \eqref{m-10} 
reduces the right side to
\[
  \Theta N(q^2)=L(q^2)(b+32t)(b+8t)(b-16t)
 -(b^2+16bt+256t^2)^2.
\]
Equations \eqref{eq:M-bt} and \eqref{eq:N-bt} now give
\eqref{eq:differential-N}.
\end{proof}

\begin{lemma}\label{lem:P3C3}
	If \(m\ge0\) and \(m\equiv1,5,\) or \(6\pmod7\),   then
\begin{equation}\label{eq:p-c-relation}
 \overline p(8m+3)\equiv(m+3)^2c(8m+3)\pmod7.
\end{equation}
\end{lemma}

\begin{proof}
By \eqref{3-13}, \eqref{t-1}, and \eqref{eq:bt},
\begin{align}
\frac{J_2^{16}}{J_4^8}
=\phi(-q^2)^8
=\phi(q)^4\phi(-q)^4
=b(b+16t). \label{s-1}
\end{align}
Since $b^3=J_1^{24}/J_2^{12}$, equation \eqref{s-1} gives
\[
\begin{aligned}
\frac{J_2^{34}}{J_1^{27}J_4^8}
&=\frac{J_2^6}{J_1^3}
  \frac{J_2^{28}}{J_1^{24}J_4^8} =\frac{J_2^6}{J_1^3}
  \frac{b+16t}{b^2}.
\end{aligned}
\]
Moreover,
\[
2q\frac{J_2^{10}J_4^8}{J_1^{19}}
=
2\frac{J_2^6}{J_1^3}\frac{t}{b^2}.
\]
Therefore, by \eqref{eq:P3} and the exact identity
\(C_3=-8J_2^6/J_1^3\),
\[
\begin{aligned}
P_3
&\equiv\frac{J_2^6}{J_1^3}\frac{b+18t}{b^2} 
 \equiv-C_3\frac{b+4t}{b^2} .
\end{aligned}
\]

By \eqref{eq:M-bt}, 
\begin{align}
	M(q^2)&\equiv b^2+2bt+4t^2. \label{m-11}
\end{align}
Since \(N(q^2)\equiv1\), we also have
\(\Theta N(q^2)\equiv0\). Hence, by \eqref{eq:differential-N}
and \eqref{m-11},
\begin{align}
	L(q^2)&\equiv (b^2+2bt+4t^2)^2.
	\label{eq:L-mod7}
\end{align}
Moreover, it follows from \eqref{eq:N-bt} and $N(q^2)\equiv1 $ that 
\[
(b+4t)(b+t)(b+5t)\equiv1 ,
\]
which is equivalent to
\begin{equation}\label{eq:cubic-expanded}
	b^3\equiv1-3b^2t-bt^2-6t^3 .
\end{equation}

Equations   \eqref{m-9}   
and \eqref{m-10} give the differential rules
\begin{align}
	\Theta b&\equiv5b\bigl(L(q^2)-b-4t\bigr),\label{eq:db}\\
	\Theta t&\equiv t\bigl(5L(q^2)+3b+3t\bigr).\label{eq:dt}
\end{align}

Since \(C_3=-8J_2^6/J_1^3\),
logarithmic differentiation gives
\[
\frac{\Theta C_3}{C_3}
=
6\Theta\log J_2-3\Theta\log J_1.
\]
Using \eqref{eq:log-J},  
we obtain
\begin{align*}
	\frac{\Theta C_3}{C_3}
	&=
	\frac12\bigl(L(q^2)-1\bigr)
	-\frac18\bigl(L(q)-1\bigr)\\
	&=
	\frac{1}{2}L(q^2)-\frac{1}{8}L(q)-\frac{3}{8}
	.
\end{align*}
Multiplying both sides by \(8C_3\) and reducing modulo \(7\), we obtain
\begin{align*}
 \Theta C_3 
	& \equiv 
	C_3(4L(q^2)- L(q)-3)
	  . 
\end{align*}
Substituting \eqref{eq:L-bt} into the above congruence 
yields  
\begin{align}
	\Theta C_3
	\equiv
	C_3\bigl(2L(q^2)+b+4t-3\bigr)
 .
	\label{eq:dC3}
\end{align}

For an integer $r$ and a polynomial $U(b,t)$, the product rule and
\eqref{eq:db}--\eqref{eq:dC3} give
\begin{equation}\label{eq:operator-rule}
 (\Theta-r)(C_3b^{-2}U)
 \equiv C_3b^{-2}
 \left(\Theta U+\bigl(6L(q^2)+4b+2t-3-r\bigr)U\right).
\end{equation}
Define 
\[
 Y:=P_3-(\Theta+3)^2C_3 
\]
and  \[H:=2L(q^2)+b+4t. \]
 From \eqref{eq:dC3},
\[
 (\Theta+3)C_3\equiv C_3H,
 \qquad
 (\Theta+3)^2C_3\equiv C_3(\Theta H+H^2).
\]
Together with the preceding expression for \(P_3\), this gives
\begin{equation}\label{eq:Y-first-reduction}
 Y\equiv-C_3b^{-2}
 \bigl(b+4t+b^2(\Theta H+H^2)\bigr)\pmod7.
\end{equation}
We now spell out the polynomial reduction used below.  By
\eqref{eq:L-mod7} and \eqref{eq:cubic-expanded},
\begin{align*}
L(q^2)&\equiv b^2t^2+b+2bt^3+t+3t^4,\\
H&\equiv2b^2t^2+3b+4bt^3+6t+6t^4.
\end{align*}
Using \eqref{eq:db} and \eqref{eq:dt} to differentiate the latter 
polynomial, and reducing every occurrence of \(b^3\) by
\eqref{eq:cubic-expanded}, we obtain 
\begin{align*}
\Theta H+H^2\equiv{}&
 2b^2+6b^2t^3+5b^2t^6+4bt+3bt^4+3bt^7 +3t^2+2t^5+6t^8.
\end{align*}
Substitution into \eqref{eq:Y-first-reduction}, followed by one more
use of \eqref{eq:cubic-expanded}, gives
\[
 b+4t+b^2(\Theta H+H^2)\equiv U_0,
\]
where we define
\begin{align*}
 U_0:={}&b^2(6t^5+2t^8)
 +b(3+3t^3+5t^6+3t^9)
 +2t+4t^4+t^7+2t^{10},\\
 U_1:={}&2b^2+bt+2t^2,\\
 U_2:={}&b^2(1+5t+t^4)
 +b(4t+2t^2+6t^5)
 +6+t^2+3t^3+3t^6,\\
 U_3:={}&b^2(2+2t+6t^4+2t^5+2t^8) +b(4+t+5t^2+4t^3+t^5+2t^6+6t^9)
 \\
 & +1+5t+2t^2+4t^3+4t^6+5t^7.
\end{align*}
Thus \(Y\equiv-C_3b^{-2}U_0\).  Continuing with
\eqref{eq:operator-rule}  and reducing by
\eqref{eq:L-mod7} and \eqref{eq:cubic-expanded} after each step,
 we obtain 
\begin{align}
 Y&\equiv-C_3b^{-2}U_0,  \label{m-15}\\
 (\Theta-4)Y&\equiv-C_3b^{-2}U_1,  \nonumber \\
 \Theta(\Theta-4)Y&\equiv-C_3b^{-2}U_2,  \nonumber \\
 (\Theta-2)\Theta(\Theta-4)Y&\equiv-C_3b^{-2}U_3, \nonumber \\
 (\Theta-3)(\Theta-2)\Theta(\Theta-4)Y&\equiv0. \label{m-16}
\end{align}
Since polynomials in $\Theta$ commute, congruences
\eqref{m-15} and \eqref{m-16} give
\begin{equation}\label{eq:annihilator}
	\Theta(\Theta-2)(\Theta-3)(\Theta-4)
	\bigl(P_3-(\Theta+3)^2C_3\bigr)\equiv0.
\end{equation}

The coefficient of $q^m$ in $Y$ is
$
 \overline p(8m+3)-(m+3)^2c(8m+3).
$
It follows from \eqref{eq:annihilator} that for $m\geq 0$,
\[
m(m-2)(m-3)(m-4)( \overline p(8m+3)-(m+3)^2c(8m+3) )\equiv 0  . 
\]
Replacing $m$ by $7m+r$ with $r=1,5,6$  in  the above congruence   proves
\eqref{eq:p-c-relation}.
\end{proof}

\begin{lemma}\label{lem:X7P7}
	If \(m\) is a positive integer and \(7\nmid m\), then
\begin{equation}\label{eq:x-p-relation}
 x(8m+7)\equiv3m^4\overline p(8m+7) .
\end{equation}
\end{lemma}

\begin{proof}
Applying \eqref{eq:freshman} to \eqref{eq:P7} gives
\begin{equation}\label{eq:P7-factor}
 P_7\equiv\frac{J_{14}^3}{J_7^4}J_1^5J_2.
\end{equation}
Moreover, \eqref{eq:X7}, \eqref{eq:P7}, and
\[
 \frac{J_2^{20}}{J_1^8J_4^8}=b+16t
\]
give
\begin{equation}\label{eq:X7-factor}
\begin{aligned}
 X_7
 &\equiv P_7\bigl(4(b+16t)+t\bigr)\\
 &\equiv P_7(4b+2t) .
\end{aligned}
\end{equation}
By  \eqref{eq:L-bt} and \eqref{eq:log-J},
\begin{equation}\label{eq:log-J15J2}
\begin{aligned}
 \Theta\log(J_1^5J_2)
 &=\frac{5L(q)+2L(q^2)-7}{24}\\
 &\equiv4L(q)+3L(q^2)\\
 &\equiv4L(q^2)+3b+5t .
\end{aligned}
\end{equation}
In view of \eqref{eq:log-J15J2}, set $V_0:=1$ and recursively define
\begin{equation}\label{eq:V-recursion}
 V_{j+1}:=\Theta V_j+\bigl(4L(q^2)+3b+5t\bigr)V_j.
\end{equation}
Then for $j\geq 0$, 
\begin{align*}
	\Theta^{j}  (J_1^5J_2) \equiv J_1^5J_2V_j.
\end{align*}
  Five applications of
\eqref{eq:V-recursion}, using 
\eqref{eq:L-mod7}--\eqref{eq:dt}, give
\begin{equation}\label{eq:V5}
 V_5\equiv b^2t^3+bt^4+2bt+2t^5.
\end{equation}
The same rules give
\[
 \frac{\Theta\bigl(J_1^5J_2(4b+2t)\bigr)}{J_1^5J_2}
 \equiv3b^2t^3+3bt^4+6bt+6t^5
 \equiv3V_5
\]
by \eqref{eq:V5}.
Therefore
\begin{equation}\label{eq:core-X7}
 \Theta\bigl(J_1^5J_2(4b+2t)\bigr)
 \equiv3\Theta^5(J_1^5J_2).
\end{equation}
Since $J_{14}^3/J_7^4$ contains only terms of the form \(q^{7n}\),
\begin{align}\label{m-20}
\Theta(J_{14}^3/J_7^4) \equiv 0  .
\end{align}
In view of \eqref{eq:P7-factor},
\eqref{eq:X7-factor}, \eqref{eq:core-X7}, and \eqref{m-20},
\begin{align}
\Theta(X_7)&\equiv \Theta( J_{14}^3/J_7^4\cdot  J_1^5J_2 (4b+2t))\nonumber\\
& \equiv
  J_1^5J_2 (4b+2t)\Theta( J_{14}^3/J_7^4)+  \frac{J_{14}^3}{J_7^4}
\Theta(  J_1^5J_2 (4b+2t)) \nonumber\\
& \equiv3\frac{J_{14}^3}{J_7^4}\Theta^5(J_1^5J_2) .\label{m-22}
\end{align}
On the other hand,
\begin{align}
\Theta^5(P_7) =\Theta^5( J_{14}^3/J_7^4\cdot J_1^5J_2) 
\equiv \frac{J_{14}^3}{J_7^4}\Theta^5(J_1^5J_2)  . \label{m-23}
\end{align}
In light of  \eqref{m-22} and \eqref{m-23},
\begin{align}
\label{eq:X7P7-differential}
\Theta X_7\equiv3\Theta^5P_7 .
\end{align}
Equation \eqref{eq:X7P7-differential} is the required differential
relation.  Comparing coefficients of $q^m$, we get
  \[
  mx(8m+7)\equiv 3m^5\overline{p}(8m+7) .
  \]
Dividing by $m$ when $7\nmid m$ gives \eqref{eq:x-p-relation}.
\end{proof}

\section{Proof of Conjecture \ref{conj-1}}

\begin{proof}
We prove the five congruences in their stated order.  First, apply
Lemma~\ref{lem:P3C3} with $m=7n+5$.  Since
\((m+3)^2\equiv1\pmod7\), we obtain
\[
 \overline p(56n+43)\equiv c(56n+43) .
\]
Substitution into \eqref{eq:OSOME} proves \eqref{c-1}.

Next, Lemmas~\ref{lem:sparse} and \ref{lem:P3C3}, with $m=7n+1$,
give
\[
 x(56n+11)\equiv0,
 \qquad
 \overline p(56n+11)\equiv2c(56n+11) .
\]
The right side of \eqref{eq:OSOMEo} therefore vanishes modulo $7$,
which proves \eqref{c-2}.

It follows from  \eqref{eq:c-eight-seven} that 
 \begin{align}
c(56n+23)=c(56n+47)=0.
\label{g-0}
\end{align}
Replacing $m$ by $7n+2$ and $7n+5$ 
 in Lemma~\ref{lem:X7P7} gives
 \begin{align}\label{g-1}
 x(56n+23)\equiv 6 \overline p(  56n+23 ) , 
 \end{align}
 and 
   \begin{align}
  x(56n+47)
\equiv 6 \overline p(56n+47).\label{g-2}
 \end{align}
Equation \eqref{eq:OSOMEo}, together with \eqref{g-0}--\eqref{g-2}, proves \eqref{c-3} and \eqref{c-4}.

Finally, Lemmas~\ref{lem:sparse} and \ref{lem:P3C3}, with $m=7n+6$,
give
\[
 x(56n+51)\equiv0,
 \qquad
 \overline p(56n+51)\equiv4c(56n+51).
\]
Replacing $n$ by $56n+51$ in  \eqref{eq:OSOMEe}
 and using the above congruences,  we get   \eqref{c-5}.
\end{proof}

\end{document}